\documentclass[reqno]{amsart} %
\usepackage{amssymb, mathdots}
\usepackage{mathabx}
\usepackage{mathrsfs}
\usepackage[utf8]{inputenc}
\usepackage{amsmath,  graphicx}
\usepackage{diagrams}
\usepackage{tikz}
\usetikzlibrary{matrix}
\usepackage{wrapfig}  
\usepackage{enumerate}
\usepackage{url, hyperref}
\DeclareSymbolFont{bbold}{U}{bbold}{m}{n}
\DeclareSymbolFontAlphabet{\mathbbold}{bbold}

\def\resto#1#2{{
#1\hskip 0.4ex\vline_{\hskip 0.2ex\raisebox{-0,2ex}
{{${\scriptstyle #2}$}}}}}

 \def\psp#1#2%
  {\mathop{}%
   \mathopen{\vphantom{#2}}^{#1}%
   \kern-\scriptspace%
    \hskip -0.3mm{#2}} 
 \def\psb#1#2%
  {\mathop{}%
   \mathopen{\vphantom{#2}}_{#1}%
   \kern-\scriptspace%
    \hskip -0.0mm{#2}} 
 \def\pscr#1#2#3%
  {\mathop{}%
   \mathopen{\vphantom{#3}}^{#1}_{#2}%
   \kern-\scriptspace%
    \hskip -0.3mm{#3}} 

\def\C{{\mathbb C}}

\def\H{{\mathbb H}}

\def\N{{\mathbb N}}

\def\P{{\mathbb P}}
\def\Q{{\mathbb Q}}
\def\R{{\mathbb R}}

\def\Z{{\mathbb Z}}

\def\union{\mathop{\bigcup}}

\def\textmap#1{\mathop{\vbox{\ialign{
                                  ##\crcr
      ${\scriptstyle\hfil\;\;#1\;\;\hfil}$\crcr
      \noalign{\kern 1pt\nointerlineskip}
      \rightarrowfill\crcr}}\;}}
\def\bigtextmap#1{\mathop{\vbox{\ialign{
                                  ##\crcr
      ${\hfil\;\;#1\;\;\hfil}$\crcr
      \noalign{\kern 1pt\nointerlineskip}
      \rightarrowfill\crcr}}\;}}
      
\newcommand{\cal}{\mathcal}
\def\textlmap#1{\mathop{\vbox{\ialign{
                                  ##\crcr
      ${\scriptstyle\hfil\;\;#1\;\;\hfil}$\crcr
      \noalign{\kern-1pt\nointerlineskip}
      \leftarrowfill\crcr}}\;}}

\def\ig{{\mathfrak i}}

\def\pg{{\mathfrak p}}
\def\qg{{\mathfrak q}}

\def\tg{{\mathfrak t}}

\def\Cg{{\mathfrak C}}

\def\Ig{{\mathfrak I}}

\def\Xg{{\mathfrak X}}
\def\Yg{{\mathfrak Y}}

\theoremstyle{remark}
\newtheorem{ex}{Example}[section]

\newtheorem{quest}{Question}
\newtheorem{conj}{Conjecture}
\newtheorem*{conjGSS}{The GSS conjecture}

\newtheorem{sz}{Satz}[section]
\theoremstyle{remark}
\newtheorem{re}[sz]{Remark} 
\theoremstyle{plain}
\newtheorem{thry}[sz]{Theorem}
\newtheorem{pr}[sz]{Proposition}
\newtheorem{co}[sz]{Corollary}
\newtheorem{dt}[sz]{Definition}
\newtheorem{lm}[sz]{Lemma}

\def\Pic{\mathrm {Pic}}

\def\im{\mathrm{im}}

\def\kod{\kappa}

\def\Def{\mathrm{Def}}
\def\Sing{\mathrm{Sing}}

\newcommand\smvee{{\hskip -0.15ex \raise 0.2ex\hbox{$\scriptscriptstyle\vee$}}}
\def\we{{\smvee \hskip -0.2ex \smvee}}
\newcommand{\longhookrightarrow}{}
\DeclareRobustCommand{\longhookrightarrow}{\lhook\joinrel\relbar\joinrel\rightarrow}

\begin{document}
 
\title [Smooth  deformations of singular contractions]{Smooth  deformations of singular contractions of class VII surfaces}
 \author{Georges Dloussky} 
  \address{Aix Marseille Université, CNRS, Centrale Marseille, I2M, UMR 7373, 13453 Marseille, France   }
\email[Georges Dloussky]{georges.dloussky@univ-amu.fr}
\author{Andrei  Teleman}
\email[Andrei  Teleman]{andrei.teleman@univ-amu.fr}
\begin{abstract}
We consider normal compact surfaces $Y$ obtained from a minimal  class VII surface  $X$  by contraction of a cycle $C$ of $r$ rational curves with $C^2<0$. Our main result states that, if the obtained cusp is smoothable, then $Y$ is globally smoothable. The proof   is based on a vanishing theorem for $H^2(\Theta_Y)$. 

If $r<b_2(X)$ any smooth small deformation of $Y$ is rational, and   if $r=b_2(X)$ 
(i.e. when $X$ is a half-Inoue surface)   any smooth small deformation of $Y$ is an Enriques surface.

 The condition ``the cusp is smoothable" in our main theorem can be checked  in terms of the intersection numbers of the cycle, using  the Looijenga conjecture (which has recently become a theorem). Therefore this is a ``decidable" condition. We prove that this condition is always satisfied if $r<b_2(X)\leq 11$. Therefore the singular surface $Y$ obtained by contracting a cycle $C$ of $r$ rational curves in  a minimal class VII surface $X$ with $r<b_2(X)\leq 11$  is always smoothable by rational surfaces.  The statement holds even for unknown class VII surfaces. 
\end{abstract}
\thanks{The authors thank M. Manetti for a fruitful exchange of mails about deformation theory. We also thank very much the unknown referee for the careful reading of the article, and for many valuable suggestions and remarks. }
\maketitle 
%

\section{Introduction}

\subsection{The results}

Class VII surfaces \cite[p. 244]{BHPV} are not classified yet. The global spherical shell (GSS) conjecture, which, if true, would complete the classification of this class, can be stated as follows:
\begin{conjGSS}\label{GSS}
Any minimal class VII surface with $b_2>0$ is a Kato surface. 	
\end{conjGSS}
Recall that a Kato surface is a minimal class VII surface  with positive $b_2$ which admits a GSS. Kato surfaces are considered to be the known surfaces in the class VII; they can be obtained by a simple two-step construction: iterated blown up of the standard ball $B\subset \C^2$  over the origin, followed by a glueing procedure (see \cite{Dl84}, \cite{DlCplxMan} for details).

This important conjecture has been stated by Nakamura in \cite{NaSug} and has been proved for surfaces with $b_2=1$ in \cite{Te1}.  Any Kato surface $X$ has $b_2(X)$ rational curves, and some of these curves form a cycle of rational curves (see section \ref{HZsection}). Therefore any Kato surface contains a cycle of $r$ rational curves with $0< r\leq b_2(X)$.  The standard approach for proving the GSS conjecture   has two steps corresponding to the following two conjectures considered   by experts to be more accessible:
\begin{conj}\label{ConjA}
Any minimal class VII surface with $b_2>0$ has a cycle of rational curves.	
\end{conj}

\begin{conj}\label{ConjB}
Any minimal class VII surface with $b_2>0$ containing   	a cycle of rational curves is a Kato surface. 
\end{conj}

Conjecture \ref{ConjA} is itself very important because, by a result a Nakamura \cite[Theorem p. 476]{NaToh}, any minimal class VII surface with $b_2>0$ containing a cycle of rational curves is a deformation in large  of a family of blown up primary Hopf surfaces. Therefore, if true, this conjecture would solve the classification of class VII surfaces up to deformation equivalence.  Moreover, this conjecture has been proved for $b_2\leq 3$, and a program for the general case has been developed in  \cite{Te2}, \cite{Te3}.

On the other hand, at this moment we have no method or approach to make progress on Conjecture \ref{ConjB}. The problem is that, although blown up primary Hopf surfaces appear to be very simple complex geometric objects, we have no method to classify or control degenerations (deformations in large)  of families of such surfaces. 

The results in this article are related to Conjecture  \ref{ConjB}, and suggest a new approach which might lead to a proof of this conjecture   for small $b_2$.  Our approach leads to   several new and surprising open questions, which (we believe) are of independent interest for both algebraic geometers and experts in complex analytic geometry. \\

Our main result gives a precise answer to the following question:
\begin{quest}
Let $X$ be a minimal class VII surface and $C\subset X$ be a cycle of rational curves with $C^2<0$. Let $Y$ be the normal surface obtained by contracting $C$ to a point. Does $Y$ admit smooth deformations? If yes, what surfaces appear as smooth deformations of such a normal surface $Y$?	
\end{quest}

This problem has been studied before in special cases by Nakamura \cite{Na80} and Looijenga  \cite{Lo81}. Using the notations and the terminology introduced in section \ref{HZsection}, our main result reads: 
\newtheorem*{th-main}{Theorem \ref{lissageglobal}} 
\begin{th-main}
Let $X$ be a minimal class VII surface,  $C\subset X$ be a  cycle of $r$ rational curves with $C^2<0$, $[c_0,\ldots,c_{r-1}]$ be its type, and  $(Y,c)$ be the singular contraction	 of $(X,C)$.   Then $r\leq b_2(X)$ and
\begin{enumerate}
\item $Y$ is smoothable if and only if the  dual  $[d_0,\ldots, d_{s-1}]$ of  $[c_0,\ldots,c_{r-1}]$ is the type of an anti-canonical cycle in a smooth rational surface which admits $\P^2$ as minimal model. This condition is always satisfied when $\sum_{i=0}^{r-1} (c_i-2)\leq 10$.  	
\item If $r<b_2(X)$, then  any smooth deformation $Y'$  of $Y$ is a rational surface with 
$$b_2(Y')=10+b_2(X)+C^2=10+r\,.$$
\item If $r=b_2(X)$, then	 $X$ is a half-Inoue surface, and  any smooth deformation $Y'$ of $Y$ is an Enriques surface.
\end{enumerate}
\end{th-main}

The condition ``the  dual  $[d_0,\ldots, d_{s-1}]$ of  $[c_0,\ldots,c_{r-1}]$ is the type of an anti-canonical cycle in a blown up $\P^2_\C$" is decidable algorithmically in terms of the intersection numbers $c_i$. Moreover, the condition $\sum_{i=0}^{r-1} (c_i-2)\leq 10$  is automatically satisfied when $r<b_2(X)\leq 11$ so, in this range, the singular surface $Y$ obtained starting with a pair $(X,C)$ as in Theorem \ref{lissageglobal}  is always smoothable by rational surfaces. 

The main ingredients  in the proof are:  
\begin{itemize}
\item The theorem of  Gross-Hacking-Keel \cite{GHK15}, and Engel \cite{En14} stating that  Looijenga's conjecture is true (see section \ref{LooijSection}). This important result gives a smoothability criterion for the cusp $(Y,c)$, regarded as an  isolated singularity.
\item A general vanishing theorem for $H^2(Y,\Theta_Y)$ (Theorem \ref{H2=0}), which will be proved in detail in section \ref{VanishSection}. 	
\end{itemize}

Compared to previous results dedicated to this problem, our Theorem \ref{lissageglobal} concerns arbitrary (including unknown) class VII surfaces.  Note also that, even within in the class of Kato surfaces, the Inoue-Hirzebruch surfaces considered in \cite{Na80} and \cite{Lo81} are very special: their isomorphism classes give isolated points in the moduli space of Kato surfaces with fixed $b_2$ (which is $b_2$-dimensional).

Let $\tg=[c_0,\dots,c_{r-1}]$ be an oriented cycle of integers, with $c_i\geq 2$ for  $0\leq i\leq r-1$ and $c_i>2$ for at least an index $i$ (see section \ref{HZsection}). A cusp $(V,c)$ will be called cusp of type $\tg$ if the type of the exceptional divisor $C$ of its minimal resolution  (endowed with a suitable orientation) is $\tg$. Suppose now that the cusp $(V,c)$ is smoothable. By the theorem of Gross-Hacking-Keel-Engel mentioned above, this happens if and only if the dual  $[d_0,\dots,d_{s-1}]$ of $\tg$ is the type of an anti-canonical cycle in a smooth rational surface which admits $\P^2$ as minimal model (see  section \ref{LooijSection}).

Let $\Yg_\tg$ be the class of singular compact surfaces $Y$   with the properties
\begin{enumerate}
\item $Y$ has a single singularity which is a cusp $c$  of type $\tg$, and $Y\setminus\{c\}$ is minimal. 
\item $a(Y)=q(Y)=0$, $\kod(Y)=-\infty$, where $a$, $q$, $\kod$ stand for the algebraic dimension, irregularity and the Kodaira dimension respectively.
 \end{enumerate}
 Theorem  \ref{lissageglobal} suggests a new strategy for proving Conjecture \ref{ConjB} for a large  subclass of VII surfaces, namely for  class VII surfaces with a cycle of $r< b_2(X)$ rational curves, satisfying the smoothability criterion of Theorem \ref{lissageglobal} (1). This subclass includes all class VII surfaces with $b_2(X)\leq 11$ which are not half-Inoue surfaces. This new strategy has two steps:
\begin{enumerate}[(a)]
\item Classify all families $(Y_z)_{z\in D}$ such that $Y_0\in\Yg_\tg$, and for $z\ne 0$ the fibre $Y_z$ is a smooth rational surface with $b_2(Y_z)=10+r$. 
\item Prove that any surface $Y_0\in\Yg_\tg$ obtained in this way  contains   $s=\sum_{i=0}^r (c_i-2)$ rational curves. 	
\end{enumerate}

Note that, if $Y_0$  contains $s$ rational curves,  its  minimal resolution $X$ will be a minimal class VII surface with $r+s$ rational curves. Since   $X$ is minimal, we have  $r+s=b_2(X)$ \cite[Corollary 2.36]{DlJAMS}, so $X$ will be  a  minimal class VII surface with $b_2(X)$ rational curves. By the main result of \cite{DOT3} such a surface is a Kato surface. Therefore this strategy will lead to a proof of Conjecture \ref{ConjB} for our subclass of class VII surfaces.

Concerning  step (a) note that  Manetti has shown \cite{Ma} that any normal degeneration $Y$ of $\P^2_\C$ is a  possibly singular  projective surface $Y$ with $q(Y)=P_n(Y)=0$ ($n>0$). Here   $P_n(Y)$ stand for the   the $n$-plurigenus of $Y$. Manetti has also proved    a classification theorem for the normal degenerations of $\P^2_\C$.

Theorem \ref{lissageglobal} shows that a degeneration  of  a family of non-minimal rational surfaces may have vanishing algebraic dimension, so may be an object exterior to the algebraic geometric category. The problem raised in step (a) concerns degenerations of families of arbitrary rational surfaces, so it appears to be more complex than Manetti's classification of degenerations of $\P^2_\C$. On the other hand step (a) does not require the classification of all possible normal degenerations, but   only of the  degenerations which belong to the class $\Yg_\tg$ (so with prescribed, very simple singularities).  \\

Theorem  \ref{lissageglobal} leads to difficult, interesting questions which are of independent interest:
\begin{quest}
Let ${\cal M}$ be the fine moduli space of polarized Enriques surfaces considered in \cite[section 5.3]{JoKr}, and $\Xg\to {\cal M}$ be the universal family. 	Classify the holomorphic maps  $h:D^*\to {\cal M}$ for which the pull-back family $(X_z)_{z\in D^*}$ converges to a surface with a single singularity, which is a cusp.
\end{quest}

\begin{quest}
Let $S_N$ be the configuration space parameterizing $N$-fold iterated blown up projective planes, and let $\Xg_N\to S_N$ be the universal family \cite[section 5.1]{DeGr}. Classify the holomorphic maps  $h:D^*\to S_N$ for which the pull-back family $(X_z)_{z\in D^*}$ converges to a  surface $Y$ with a single singularity, which is a cusp. 	
\end{quest}

Note that the curve configuration of the degeneration $Y$ might be quite complicated: for instance, if $X$ is an intermediate Kato surface (a Kato surface with a unique cycle and trees of rational curves intersecting the cycle) $Y$ will contain a union of trees of rational curves passing through the singularity.

Even in the case when $Y$ is the contraction of a known surface, in general we cannot expect the rational surfaces $X_z$ to possess an anti-canonical divisor.

\subsection{Cycles of rational curves on class VII surfaces. The Hirzebruch-Zagier duality}\label{HZsection}

We recall that a cycle of rational curves in a complex surface $X$ is an effective divisor $C\subset X$ which is either a nodal rational curve, or $C=C_0+C_1$ where $C_i$ are two smooth rational curves which intersect transversally in two points, or can be written as $C=\sum_{i\in\Z_r}C_i$ with $r\geq 3$ and $C_i$ are smooth rational curves such that
\begin{equation}\label{cycledef}
C_i\cdot C_j=\left\{
\begin{array}{ccc}
1 & \rm if & j=i\pm 1 \hbox{ mod } r\\	
0 & \rm if & j\ne i\pm 1 \hbox{ mod } r
\end{array}\right.\,.	
\end{equation} 
For $r\geq 3$ we will use only indexing functions 
$$\Z_r\to  \mathrm{Irr}(C):=\hbox{  the set of irreducible components of }C$$
for which (\ref{cycledef}) holds.  
%

\begin{dt} An oriented cycle  of  rational curves is a cycle of rational curves endowed with a generator   $\chi$ of $H_1(C,\Z)$ (or, equivalently, 	endowed with an orientation of the line	$H_1(C,\R)$).
\end{dt}

We recall that the class VII of surfaces is defined by
$$\mathrm{VII}:=\{X \hbox{ complex surface }|\ b_1(X)=1,\ \kod(X)=-\infty\}\,.
$$
The line $H^1(X,\R)\simeq\R$ of a class VII surface can be canonically oriented. Indeed, let $g$ be a Gauduchon metric on $X\in \mathrm{VII}$. The formula
$$[\eta]_\mathrm{DR}\mapsto -\int_X d^c\eta\wedge\omega_g
$$
defines an isomorphism $\H^1(X,\R)\textmap{\simeq}\R$, and the orientation of $H^1(X,\R)$ induced by this orientation is independent of $g$.  It is well-known that, for a cycle of rational curves $C\subset X$ the induced map $H_1(C,\R)\to H_1(X,\R)$ is always an isomorphism. Therefore 
\begin{re} A cycle of rational curves in a class VII surface is canonically oriented.	
\end{re}
Note that this remark also holds for the unknown class VII surfaces. For a cycle $C$ in a class VII surface we will always use an indexing function which is compatible with  this canonical orientation.   
%

Let $M$ be a set, and $r\in\N^*$.   We recall that an oriented $r$-cycle in $M$ is an orbit with respect to the natural right $\Z_r$-action on the set 
$$M^{\Z_r}:=\{\gamma:\Z_r\to M\}$$
given by composition with translations in $\Z_r$. The orbit of $(m_0,\dots,m_{r-1})\in M^{\Z_r}$ will be denoted  $[m_0,\dots,m_{r-1}]$. The set of oriented $r$-cycles in $M$ will be denoted by $M^{\Z_r}/\Z_r$.

\begin{dt} \label{TypeCycle} Let $C=\sum_{i\in\Z_r} C_i$ be an oriented cycle of $r$ rational curves, where the indexing function $i\mapsto C_i$ is compatible with the orientation.  The type of $C$ is the element $[c_0,\ldots,c_{r-1}]\in  \Z^{\Z_r}/\Z_r$ given by
$$[c_0,\ldots,c_{r-1}]=\left\{
\begin{array}{ccc}
[-C_0^2,\ldots,-C_{r-1}^2] &{\rm if}& r\ge 2\\ \quad [2-C_0^2] &{\rm if}& r=1	
\end{array}\right.\,.
$$
\end{dt}
The separate definition in the case $r=1$ is chosen such that the following universal formula holds for any cycle:
\begin{equation}\label{UnivForm}
-C^2=\sum_{i=0}^{r-1} (c_i-2)\,.	
\end{equation}

\vspace{3mm}

Let now $X$ be a minimal class VII surface with $b_2(X)>0$,  let $C$ be a cycle of rational curves in $X$, and let $[c_0,\dots,c_{r-1}]$ be its type. Then $c_i\geq 2$ for all $0\leq i\leq r-1$. Moreover, using (\ref{UnivForm}) we see that   $[c_0,\dots,c_{r-1}]=[2,\dots,2]$ if and only if $C^2=0$. If $X$ has such a cycle, then $r=b_2(X)$, and $X$ is an Enoki surface, i.e. it is biholomorphic with a compactification of an affine line bundle on an elliptic curve \cite{Eno}. In this article we are interested in minimal class VII surfaces $X$ having a cycle of rational curves $C$ with $C^2<0$. The type of such a cycle is an element $[c_0,\dots,c_{r-1}]\in  \Z^{\Z_r}/\Z_r$, where $c_i\geq 2$ for all $i$, but there exists $j$ such that  $c_j\geq 3$.     

Therefore we define
\begin{equation}\label{Tr}
\begin{split}
{\cal T}_r&:=\left\{[c_0,\dots,c_{r-1}]\in  (\Z_{\geq 2})^{\Z_r}/\Z_r\ \vline \ \exists j\in\Z_r \hbox{ such that } c_j\geq 3\right\},\\
 {\cal T}&:=\bigcup_{r\in\N^*}{\cal T}_r\,,
 \end{split}	
\end{equation}
and we note that the type of a cycle of rational curves with negative self-intersection in a minimal class VII surface is always an element of ${\cal T}$.  An important role in our article will be played by the following
\begin{re}
The set  ${\cal T}$ comes with a natural involution, called the Hirzebruch-Zagier duality.
\end{re} 
This duality can be defined as follows: for an element $\tg=[c_0,\dots,c_{r-1}]\in {\cal T}$ its dual $\tg^*$ is $[d_0,\dots,d_{s-1}]$ where
$$s:=\sum_{j=0}^{r-1} (c_j-2)=-C^2\,,$$
and $[d_0,\dots,d_{s-1}]$ is defined in the following way: Put $k:= \#\{j\in\Z_r|\ c_j\geq 3\}$, and write $[c_0,\dots,c_{r-1}]$ in the  form
$$[c_0,\dots,c_{r-1}]=[a_0+2,\underbrace{2,\dots,2}_{b_0-1},\dots,a_{k-1}+2,\underbrace{2,\dots,2}_{b_{k-1}-1}]\,,
$$
where $a_i$, $b_i\in\N^*$. Then
$$[d_0,\dots,d_{s-1}]=[\underbrace{2,\dots,2}_{a_0-1},b_0+2,\dots,\underbrace{2,\dots,2}_{a_{k-1}-1}, b_{k-1}+2]\,.
$$
In other words, the dual cycle $[d_0,\dots,d_{s-1}]$ is constructed by replacing any element $c_j\geq 3$ in the original  cycle by the (possibly empty) sequence $\underbrace{2,\dots,2}_{c_{j}-3}$, and any maximal (possibly empty) sequence of the form $\underbrace{2,\dots,2}_{l}$ in the original cycle by the single element $l+3$.\\

The Hirzebruch-Zagier duality has been discovered as a geometric property of the  curves on Inoue-Hirzebruch surfaces  \cite{Hi73}, \cite{HiZ}, \cite{Na80} \cite{Dl88}: 
\begin{re}
An Inoue-Hirzebruch surface has two cycles of rational curves whose types are dual to each other. 	
\end{re}

The following remark summarizes several general, basic properties of cycles with negative self-intersection on minimal class VII surfaces (see \cite[Theorem 9.2]{NaInv}). In particular we see that a half-Inoue surface $X$ has a cycle $C$ of $b_2(X)=-C^2$ rational curves whose type is self-dual.
\begin{re}\label{HalfInoue}  Let $X$ be minimal class VII surface, and $C\subset X$ be a  cycle of $r$ rational curves with $C^2<0$, and $[c_0,\dots,c_{r-1}]$ be its type. Then
\begin{enumerate}
	\item 	$r\leq b_2(X)$.
\item The natural morphism $H_1(C,\Z)\to H_1(X,\Z)$ is a monomorphism, and 
 $$[H_1(X,\Z):H_1(C,\Z)]\in\{1,2\}\,.$$
\item One has $r< b_2(X)$ if and only if $[H_1(X,\Z):H_1(C,\Z)]=1$.
\item \label{r=b2} One has $r= b_2(X)$ if and only if $X$ is a half-Inoue surface. Such a surface has the properties:
\begin{enumerate}
\item $[H_1(X,\Z):H_1(C,\Z)]=2$.
\item The type $[c_0,\ldots,c_{r-1}]$ of $C$ coincides with its dual, in particular $$-C^2=r=b_2(X)\,.$$
\item \label{order2} The isomorphism class $[{\cal K}_X(C)]$ is an order 2 element of $\Pic^0(X)$.
\end{enumerate}
\end{enumerate}
\end{re}

\section{A vanishing theorem for singular contractions of class VII surfaces}
\label{VanishSection}

For a complex space $Y$ we will denote by $\Omega_Y$ its sheaf of Kähler differentials, and by $\Theta_Y$ its dual. If $Y$ is smooth of dimension $n$ then both sheaves are locally free of rank  $n$, and can be identified with the sheaves of sections of the cotangent, respectively tangent bundle.\\

Let $X$ be a minimal class VII surface, and let $C\subset X$ be a  cycle of rational curves  with $C^2<0$. The intersection form of such a cycle is negative definite. Let $\pi:X\to Y$ be the projection onto the singular surface  obtained by contracting $C$, and let $c\in Y$ be the isolated normal singularity corresponding to $C$. This singularity is a cusp, in particular it is minimal elliptic and Gorenstein \cite[Theorem 7.6.6, Definition 7.6.9, p. 153]{Ish}, so the dualising sheaf $\omega_Y$ is invertible. 
\begin{lm}\label{FirstIsos} For a sufficiently small open neighbourhood $U$ of $C$ the following holds
\begin{enumerate}
\item \label{Ktrivial} The line bundle ${\cal K}_U(C)$ is trivial.
\item \label{Krest} For any $n\in\Z$ the restriction morphism ${\cal K}_X^{\otimes n}(nC)(U)\to 	{\cal K}_X^{\otimes n}(U\setminus C)$ is an isomorphism.
\item  \label{Omega} The restriction morphism $\Omega_X(U) \to	\Omega_X(U\setminus C)$ is an isomorphism.
\end{enumerate}	
\end{lm}
\begin{proof}
There exists an open neighbourhood $U$ of $C$ in $X$ which can be identified with a neighbourhood $\tilde U$ of a cycle $\tilde C$ in an Inoue-Hirzebruch surface $\tilde X$ such that $C$ corresponds to $\tilde C$. This follows from  \cite[p. 139]{Lau}, taking into account that any type  $\tg\in {\cal T}$ coincides with the type of a cycle of an Inoue-Hirzebruch surface (see also \cite[proof of Lemma 4.3, p. 409]{NaInv}).  We suppose that $\tilde U$ does not intersect the second cycle $\tilde D$ of $\tilde X$. Since ${\cal K}_{\tilde X}\simeq {\cal O}_{\tilde X}(-\tilde C-\tilde D)$ \cite[section 6]{NaInv}, we obtain a nowhere vanishing section $\eta_0\in {\cal K}_U(C)$, which proves (1).   

For (2) note that the restriction morphism ${\cal K}_X^{\otimes n}(nC)(U)\to 	{\cal K}_X^{\otimes n}(U\setminus C)$ is obviously injective. Let $\lambda\in {\cal K}_X^{\otimes n}(U\setminus C)$, and let $\varphi\in {\cal O}(U\setminus C)$ be such that 
$$\lambda=\varphi\resto{\eta_0^{\otimes n}}{U\setminus C}\,.$$
Since   $c$ is an isolated normal singularity there exists  a  holomorphic extension  $\tilde\varphi$ of $\varphi$ to $U$ \cite[Proposition p. 119]{Fi}.  Then $\tilde\lambda:=\tilde\varphi  \eta_0^{\otimes n}\in {\cal K}_X^{\otimes n}(nC)(U)$, and its restriction to $U\setminus C$ is $\lambda$. 

For (3) note first that the restriction morphism $\Omega_X(U)\to 	\Omega_X(U\setminus C)$ is obviously injective.  For the surjectivity, let $\eta\in \Omega_X(U\setminus C)=H^0(U\setminus C,\Omega^1)$. By (2) the differential 
$$d\eta\in H^0(U\setminus C,\Omega^2)= {\cal K}_X(U\setminus C)$$
  extends to a section $\kappa\in {\cal K}_X(C)(U)$, so $d\eta\in H^0(U,\Omega^2(\log(C)))$, so the class of $d\eta$ in the quotient $H^0(U\setminus C,\Omega^2)/H^0(U,\Omega^2(\log(C)))$ vanishes. 
By \cite[Corollary 1.4]{SS} this implies that $\eta$ extends on $U$.
\end{proof}

\begin{thry} \label{piTh} Let $X$ be a minimal class VII surface, and $C\subset X$ a  cycle  with $C^2<0$. Let $X':=X\setminus C$, $Y':=Y\setminus\{c\}$, and let $\pi':X'\to Y'$ be the induced biholomorphism.  
\begin{enumerate}
\item \label{omegaY} Let $n\in\Z$.
\begin{enumerate}[(a)]
\item The obvious isomorphism $\omega_{Y'}^{\otimes n}\textmap{\simeq} \pi'_*({\cal K}_{X'}^{\otimes n})$ extends to an isomorphism 
$$\omega_Y^{\otimes n} \textmap{\simeq}\pi_*({\cal K}_X^{\otimes n}(nC)) \,.$$
\item  \label{piomegaY} The obvious isomorphism $\pi'^*(\omega_{Y'}^{\otimes n})\textmap{\simeq} {\cal K}_{X'}^{\otimes n}$ extends to an isomorphism 
$$\pi^*(\omega_Y^{\otimes n})\textmap{\simeq} {\cal K}_X^{\otimes n}(nC)\,.$$
\end{enumerate}
\item \label{piOmegaX} The obvious isomorphism $\pi'_*(\Omega_{X'})\textmap{\simeq} \Omega_{Y'}$ extends to an isomorphism 
$$\pi_*(\Omega_X)=\Omega_Y^{\we}\,.$$
\end{enumerate}

\end{thry}

\begin{proof} (1)(a) The sheaf $\omega_Y^{\otimes n}$ is invertible, in particular reflexive. Using  \cite[Proposition 5.29, p. 141]{PR} it follows that, for any open neighbourhood $V$ of $c$, the restriction  morphism $\omega_Y^{\otimes n}(V)\to \omega_Y^{\otimes n}(V\setminus\{c\})$ is an isomorphism. On the other hand, one has an obvious isomorphism $\omega_Y^{\otimes n}(V\setminus\{c\})\simeq{\cal K}_X^{\otimes n}(\pi^{-1}(V)\setminus C)$. Using Lemma \ref{FirstIsos} (\ref{Krest}) we obtain, for any sufficiently small open neighbourhood $V$ of $c$, an isomorphism $\omega_Y^{\otimes n}(V)\simeq {\cal K}_X^{\otimes n}(nC)(\pi^{-1}(V))$.

(1)(b) By Lemma \ref{FirstIsos} (\ref{Ktrivial}) we know that the line bundle ${\cal K}_X^{\otimes n}(nC)$ is trivial around $C$. Therefore   the canonical morphism
$$\pi^*\big(\pi_*({\cal K}_X^{\otimes n}(nC))\big)\to {\cal K}_X^{\otimes n}(nC)
$$
is an isomorphism.  But $\pi_*({\cal K}_X^{\otimes n}(nC))=\omega_Y^{\otimes n}$ by (1)(a).

(2) The bidual $\Omega_Y^{\we}$ is a reflexive sheaf \cite[Corollary 1.2]{HaSheaves}. Using \cite[Proposition 5.29]{PR} it follows that, for any  open neighbourhood $V$ of $c$, the restriction morphism
$$\Omega_Y^{\we}(V)\to \Omega_Y^{\we}(V\setminus\{c\})=\Omega_X(\pi^{-1}(V)\setminus C)
$$
is an isomorphism. But, by Lemma \ref{FirstIsos} (\ref{Omega}), for sufficiently small $V$, one has $\Omega_X(\pi^{-1}(V)\setminus C)=\Omega_X(\pi^{-1}(V))$.
\end{proof}

\begin{co} \label{h2Y} One has $h^2({\cal O}_Y)=0$.
\end{co}
\begin{proof}
Using Serre duality for Cohen-Macauley complex spaces  \cite[Theorem 4.15a p. 171]{Pe}, we obtain $h^2({\cal O}_Y)=h^0(\omega_Y)$.  By Theorem \ref{piTh} (\ref{omegaY})  and the classical Serre duality theorem we have 
$$h^0(\omega_Y)=h^0({\cal K}_X(C))=h^2({\cal O}_X(-C))\,.$$
 The exact sequence
$$\to H^1({\cal O}_X)\textmap{\simeq} H^1({\cal O}_C)\to H^2({\cal O}_X(-C))\to H^2({\cal O}_X)=0
$$
shows that $h^2({\cal O}_X(-C))=0$.
\end{proof}

\begin{co} \label{Iso} Let $X$ be a minimal class VII surface, and $C\subset X$ a  cycle  with $C^2<0$. One has an isomorphism
$$\pi_*(\Omega_X\otimes {\cal K}_X(C))\simeq \Omega_Y^{\we}\otimes \omega_Y\,.
$$
In particular $h^0(\Omega_Y^{\we}\otimes \omega_Y)= h^0(\Omega_X\otimes {\cal K}_X(C))$.
\end{co}
\begin{proof}
Using Theorem \ref{piTh}(\ref{piomegaY}), Theorem \ref{piTh} (\ref{piOmegaX}) and the projection formula, we obtain
$$\pi_*(\Omega_X\otimes {\cal K}_X(C))\simeq \pi_*(\Omega_X\otimes \pi^*(\omega_Y))\simeq \pi_*(\Omega_X)\otimes\omega_Y\simeq \Omega_Y^{\we}\otimes\omega_Y\,.
$$
The projection formula applies because $\omega_Y$ is locally free \cite[Exercice II.5.1, p. 124]{HaBook}. 
\end{proof}
\begin{thry}\label{H2=0}
Let $X$ be a minimal class VII surface, and $C\subset X$ be a  cycle  with $C^2<0$, and let $Y$ be the singular surface obtained by contracting $C$. One has $H^2(Y,\Theta_Y)=0$.
\end{thry}
\begin{proof}
Using Serre duality for Cohen-Macauley complex spaces  \cite[Theorem 4.15a p. 171]{Pe}, it follows that $h^2(Y,\Theta_Y)=h^0(Y,\Theta_Y^\smvee\otimes\omega_Y)=h^0(Y,\Omega^{\we}_Y\otimes\omega_Y)$. On the other hand, by Corollary \ref{Iso}, one has $h^0(Y,\Omega^{\we}_Y\otimes\omega_Y)=h^0(X,\Omega_X\otimes {\cal K}_X(C))$. Using Serre duality on the smooth surface $X$ we obtain   $h^2(Y,\Theta_Y)=h^2(X,\Theta_X(-C))$. Therefore it suffices to prove that $h^2(X,\Theta_X(-C))=0$. 

Denote by $\Theta_{X,C}$ the restriction to $C$ of the tangent  bundle of $X$, considered as a sheaf on $X$.  The  cohomology long exact sequence associated with the short exact sequence
$$0\to \Theta_X(-C)\to \Theta_X\textmap{r} \Theta_{X,C}\to 0
$$	
contains the segment
$$H^1(X,\Theta_X)\textmap{H^1(r)} H^1(X,\Theta_{X,C})\to H^2(X,\Theta_X(-C))\to H^2(X,\Theta_X)\,.
$$
One has $H^2(X,\Theta_X)=0$ \cite[Theorem 1.2, p. 478 ]{NaToh}, so it suffices to prove that the canonical morphism 
$$H^1(r):H^1(X,\Theta_X)\to H^1(X,\Theta_{X,C})$$
 is surjective.  
 
 Following the notation introduced in \cite[Definitions 4.2, p. 409]{NaInv}, let $J_C$ be the quotient sheaf $\Theta_X/\Theta_X(-\log(C))$.  Taking $E=C$ in \cite[Theorem 1.3, p. 478]{NaToh} we get $H^2(X,\Theta_X(-\log(C)))=0$, so using the long exact cohomology sequence associated with the short exact sheaf sequence 
 $$0\to \Theta_X(-\log(C))\to  \Theta_X\textmap{u} J_C\to 0\,,$$
we see that   $H^1(u):H^1(X,\Theta_X)\to H^1(J_C)$ is surjective.
The inclusion $\Theta_X(-C)\hookrightarrow \Theta_X(-\log(C))$ gives  a sheaf epimorphism $\rho:\Theta_{X,C}\to J_C$, and one obviously has  $u=\rho\circ r$, so $H^1(u)=H^1(\rho)\circ H^1(r)$. We know that $H^1(u)$ is surjective.  Since  $H^1(\rho):H^1(X,\Theta_{X,C})\to H^1(J_C)$  is an isomorphism  \cite[Lemma 4.3, p. 409]{NaInv}, it follows that $H^1(r)$ is an epimorphism as claimed.
 \end{proof}
 
 \section{Globalizing local smoothing}
 
 \subsection{Smoothing components}
 
Let $Y$ be a compact complex space, and let $(U,u_0)$ be the base of a semi-universal deformation $Y\stackrel{i_0}{\longhookrightarrow} {\cal Y}\textmap{p} U$ of $Y$. An irreducible component $(W,u_0)$ of the germ $(U,u_0)$ is called a smoothing component for $Y$, if $p^{-1}(w)$ is smooth for generic $w\in W$.  A compact complex space is $Y$ smoothable if it admits smooth small deformations, i.e. if the base of a semi-universal deformation of $Y$ has a  smoothing irreducible component. These concepts extend in a natural way for isolated singularities as follows:

Let $(Y,c)$ be connected complex space with a single singularity $c$, and let
\begin{equation}\label{versalY}
(Y,c)\stackrel{j_0}{\longhookrightarrow} ({\cal V},j_0(c))\textmap{q} (V,v_0)	
\end{equation}
be a semi-universal deformation of the singularity $(Y,c)$ with base $(V,v_0)$.
Choose  ${\cal V}$ sufficiently small such that $q$ is defined on ${\cal V}$, is flat on ${\cal V}$, and takes values in $V$. Let $\Sing(q)\subset {\cal V}$ be the analytic subset of points of ${\cal V}$ where $q$ is not regular \cite[Definition 1.112, Corollary 116]{GLS}. Using the flatness of $q$ and \cite[Theorem 1.115]{GLS} it follows that
$$\Sing(q)\cap q^{-1}(v)=\Sing(q^{-1}(v))
$$
for any $v\in V$. Put $\qg:=\resto{q}{\Sing(q)}$.
Since $c$ is an isolated singularity, 
the fibre of the germ morphism 
$$(\Sing(q),j_0(c))\to (V,v_0)$$
defined by $\qg$ reduces to $\{j_0(c)\}$, so, by \cite[Definition 1.69, Proposition 1.70]{GLS}, it follows that this germ morphism is finite.  Therefore, we may choose ${\cal V}$, $V$ such that the morphism $\qg:\Sing(q)\to V$ is finite, in particular proper.  With this choice the image $\Sigma:=\im(\qg)=q(\Sing(q))$ will be an analytic set of $V$, and the germ $(\Sigma,v_0)$ is well defined. Note that only the irreducible components of $\Sing(q)$ which contain $j_0(c)$ contribute to this germ. The irreducible components of the germ $(\Sigma,v_0)$ have the form $(q(S),v_0)$, where $(S,j_0(c))$ is an irreducible component of the  germ of analytic set  $(\Sing(q),j_0(c))$. 
\begin{dt} An irreducible component $(Z,v_0)$ of the germ $(V,v_0)$ is called a smoothing component  for the singularity $(Y,c)$, if it is not contained in $(\Sigma,v_0)$, i.e. if   any neighborhood of $v_0$ in $Z$ contains points $v$ for which $q^{-1}(v)$ is smooth.  The isolated singularity $(Y,c)$ is called smoothable if it admits smoothing components.  \end{dt}

Suppose now that $Y$ is compact, and let  
\begin{equation}\label{defYnewnew}
Y\stackrel{i_0}{\longhookrightarrow} {\cal Y}\textmap{p} U	
\end{equation}
be a deformation of $Y$ of base $(U,u_0)$, where $i_0$ identifies $Y$ with the fibre $p^{-1}(u_0)$.

 The germ morphism 
$$({\cal Y},i_0(c))\to (U,u_0)$$
is flat, so it can be regarded as a deformation of the germ $(Y,c)$ with base $(U,u_0)$. Using the universal property of a semi-universal deformation, we obtain a commutative diagram
\begin{equation}\label{defdef}
\begin{diagram}
({\cal Y},i_0(c))&\rTo^{} & ({\cal V},j_0(c))\\	
\dTo^{p}& &\dTo_{q}\\
(U,u_0) &\rTo^{} &(V,v_0)
\end{diagram}
\end{equation}
in the category of germs, which induces a germ isomorphism
$$({\cal Y},i_0(c))\to \big(U\times_V {\cal V},(i_0(c),j_0(c))\big)
$$
over $(U,u_0)$. Therefore, supposing that $U$ and ${\cal V}$ are sufficiently small, there exists an open neighborhood ${\cal Z}$ of $i_0(c)$ in ${\cal Y}$, morphisms $\phi:U\to V$, $\Phi:{\cal Z}\to {\cal V}$, such that $\phi(u_0)=v_0$, $\Phi(i_0(c))=j_0(c)$, $q\circ \Phi=\phi\circ p$, and the induced morphism $\alpha:{\cal Z}\to U\times_V {\cal V}$ is a biholomorphism.  Therefore, for any $u\in U$ the fibre $p^{-1}(u)\cap {\cal Z}$ is identified with $q^{-1}(\phi(u))$ via $\Phi$. It follows that $\Sing(\resto{p}{\cal Z})=\Phi^{-1}(\Sing(q))$,
which implies that 
$$p(\Sing(\resto{p}{\cal Z}))=\phi^{-1}(q(\Sing(q)))=\phi^{-1}(\Sigma)\,.
$$
 Put $\pg:=\resto{p}{\Sing(p)}$. The fibre $\pg^{-1}(u_0)$ coincides with the singleton $\{i_0(c)\}$, so ${\cal Z}\cap\Sing(p)$ is a neighborhood of $\pi^{-1}(u_0)$ in $\Sing(p)$.  But $\pi$ is proper (because $p$ is proper and $\Sing(p)$ is closed in ${\cal Y}$). Therefore for any sufficiently small open neighborhood   $U'$ of $u_0$ in $U$ one has 
 $$p^{-1}(U')\cap \Sing(p)\subset {\cal Z}\cap\Sing(p)=\Sing(\resto{p}{\cal Z})\,.$$
  This shows that
 \begin{lm}\label{EqSmFib}
 For any sufficiently small open neighborhood   $U'$ of $u_0$ in $U$,  the following conditions are equivalent for a point $u\in U'$ :
\begin{enumerate}
\item The fibre $p^{-1}(u)$ is singular.
\item The intersection 	$p^{-1}(u)\cap {\cal Z}$ is singular.
\item $\phi(u)\in\Sigma$.
\end{enumerate}

 \end{lm}
 
 \subsection{From local smoothing to global smoothing}
 
 Using Lemma 	\ref{EqSmFib}, \cite[Theorem 1.117]{GLS}, and \cite[Theorem 1.117(2)]{GLS} one obtains:
 
\begin{pr}\label{LocGlobSM}
Let $Y$ be a compact complex space with a single singularity $c$, let 
$$(Y,c)\stackrel{j_0}{\longhookrightarrow} ({\cal V},j_0(c))\textmap{q} (V,v_0)
$$
be  a semi-universal deformation of the singularity $(Y,c)$, and let  
\begin{equation}\label{defYnew}
Y\stackrel{i_0}{\longhookrightarrow} {\cal Y}\textmap{p} U	
\end{equation}
be a semi-universal deformation of $Y$ of base $(U,u_0)$. Suppose that the induced germ morphism $\phi:(U,u_0)\to (V,v_0)$ is regular \cite[Definition 1.112]{GLS}. 
\begin{enumerate} 
\item \label{first} $Y$ is smoothable if and only if the isolated singularity $(Y,c)$ is smoothable.
\item If  the germ $({\cal V},j(c_0))$ is smooth, then  the germs $(V,v_0)$, $(U,u_0)$, $({\cal Y},i_0(c))$ are all smooth, and ${\cal Y}$ is smooth around the fibre $i_0(Y)=p^{-1}(u_0)$.
\end{enumerate}
\end{pr}

Note that Proposition  \ref{LocGlobSM} (2) applies in particular to the cusps which are complete intersections \cite{Ka}.

Combining the first statement of Proposition \ref{LocGlobSM} with Theorem \ref{H2=0} and  \cite[Lemma 1, p. 93]{Ma} we obtain:
 \begin{pr}\label{prphiregular} Let $X$ be a minimal class VII surface,  $C\subset X$ be a  cycle  with $C^2<0$, and let $Y$ be the singular surface obtained by contracting $C$.
 Let $\Def(Y)$ be the base of a  semi-universal deformation of $Y$, and 	$\Def(Y,c)$ be the base of a semi-universal deformation of the isolated singularity $(Y,c)$. The natural germ morphism $\phi:\Def(Y)\to \Def(Y,c)$ is regular. 
 \end{pr}
 Propositions \ref{LocGlobSM}, \ref{prphiregular} now give:
 \begin{co}\label{LocGlobSmCo}
 Let $X$ be a minimal class VII surface,  $C\subset X$ be a  cycle  with $C^2<0$, and let $Y$ be the singular surface obtained by contracting $C$. Then
 \begin{enumerate}
 \item  $Y$ is smoothable if and only if the singularity $(Y,c)$ is smoothable. \item If the total space $({\cal V},j_0(c))$  of the semi-universal deformation of the singularity $(Y,c)$ is smooth, then the germs $\Def(Y,c)$ and $\Def(Y)$ are smooth, and the total space ${\cal Y}$ of the semi-universal deformation  of $Y$ is smooth in a neighborhood of $Y$.
 \end{enumerate}
	
 \end{co}

\section{The Looijenga conjecture and local smoothing of cusps}
\label{LooijSection}

A necessary condition for smoothing cusps is given by J. Wahl \cite[Corollary 4.6, Theorem 5.6]{Wa}:
\begin{pr}  \label{Wahl81} If a cusp $(Y, c)$ is smoothable, then all the smoothing components of $(Y, c)$ have the same dimension $\# \mathrm{Irr}(C)+10 + C^2$.
In particular, 
$$\# \mathrm{Irr}(C)+10 + C^2>0\,.$$ 
\end{pr}

In \cite[(2.11) p. 311]{Lo81} E. Looijenga suggests the following conjecture, which gives a necessary and sufficient smoothability condition: 
\begin{conj} Let $(Y, c)$ be a cusp. Endow the exceptional divisor $C$ of a minimal resolution with an orientation, and suppose that the type $[c_0,\ldots,c_{r-1}]$ of the oriented cycle $C$ belongs to ${\cal T}$.   The singularity $(Y, c)$ is smoothable if and only if there exists a smooth rational surface with an anti-canonical cycle  $D$ whose type $[d_0,\ldots, d_{s-1}]$ is the Hirzebruch-Zagier dual   of $[c_0,\ldots,c_{r-1}]$ (see section \ref{HZsection}).
\end{conj}
 Looijenga's conjecture has been proved by M. Gross, P. Hacking, and S. Keel \cite[Theorem 0.5]{GHK15}; a different proof has been given later  by P. Engel \cite{En14}. Therefore we have
 \begin{thry}\label{ThLooijConj}[Gross-Hacking-Keel-Engel]
  Looijenga's conjecture is true.	
 \end{thry}

Using this result the smoothable condition can be checked in an explicit, algorithmic way:  indeed, the  Hirzebruch-Zagier dual type  $[d_0,\ldots, d_{s-1}]$ can be computed easily as explained in \ref{HZsection}. Moreover, as the theorem below shows, the condition ``$[d_0,\ldots, d_{s-1}]$ is the type of an anti-canonical cycle  $D$ in a rational surface" is equivalent to  ``$[d_0,\ldots, d_{s-1}]$ is the type of an anti-canonical cycle  $D$ in a blown up projective plane", and can  be checked algorithmically. Recall from section \ref{HZsection} that we are interested only in cycles $C$ whose type $[c_0,\ldots,c_{r-1}]$ belongs to ${\cal T}$, and the Hirzebruch-Zagier duality defines an involution on this set.
\begin{thry}\label{MinMod} 
Let $S$ be a smooth rational surface with an oriented anti-canonical cycle of rational curves $D$ whose type $[d_0,\ldots, d_{s-1}]$ belongs to ${\cal T}$. Then $S$ admits $\P^2$ as minimal model.
\end{thry}
\begin{proof}
For $s\geq 4$ this follows from \cite[Corollary 2.2]{McE90}, which is consequence of a theorem of Miranda \cite[Theorem 2.1]{McE90}. Note that the condition $[d_0,\ldots, d_{s-1}]\in {\cal T}$ is essential. Indeed, $\P^1\times\P^1$ contains the anti-canonical cycle $\P^1\times\{0,\infty\}\cup \{0,\infty\}\times\P^1$, but it cannot be contracted to $\P^2$.

For $s\in\{1,3\}$ the claim follows directly from \cite[Theorem 1.1]{Lo81}. Suppose now   $s=2$. The same theorem shows that $S$ contains a finite union $E$ of disjoint exceptional curves, each having no component in common with $D$, such that contracting these curves gives  the pair  $(\P^1\times\P^1, \bar D)$, where  $\bar D=\bar D_0+\bar D_1$ is the union of two distinct rational curves   of bidegree (1,1) and $\bar D_i$ is the image of $D_i$.  Since  $[d_0,d_{1}]\in {\cal T}$	and $\bar D_i^2=2$, $E$ is not empty. But  blowing up a point of $\P^1\times\P^1$  gives already a surface which admits $\P^2$ as minimal model.
\end{proof}


Taking into account Theorem \ref{MinMod}, and \cite[Lemmata 3.2, 3.3, 3.4 ]{FrMi} we obtain the following important result which shows that the condition ``$[d_0,\ldots, d_{s-1}]$ is the type of an anti-canonical cycle on a smooth rational surface"  intervening in Looijenga's conjecture is decidable algorithmically: 
\begin{co}
Let $S$ be a smooth rational surface with an oriented anti-canonical cycle of rational curves $D$ of type  $[d_0,\ldots, d_{k-1}]\in {\cal T}$.  There exists a finite sequence  
$$(S^k,D^k)_{1\leq k\leq m}$$
 of smooth rational surfaces  with an	anti-canonical cycles such that 
\begin{enumerate}
\item $(S^m,D^m)\simeq (S,D)$.
\item $S^1=\P^2$ and $D^1$ is either the union of three general lines, or the union of a conic and a line in general position, or a cubic with a nodal singularity.
\item For $1\leq k\leq m-1$ the pair $(S^{k+1},D^{k+1})$ is obtained from $(S^{k},D^{k})$ as  follows:
\begin{enumerate}
\item $S^{k+1}$ is the blown up of $S^{k}$ at a point $p^{k}\in D^{k}$.
\item $D^{k+1}$ is the strict transform $\widetilde D^{k}$ of $D^k$ if $p^{k}\in D^{k}\setminus\Sing(D^{k})$, and $D^{k+1}=\widetilde D^{k}+E^{k+1}$ if $p^{k}\in  \Sing(D^{k})$, where $E^{k+1}$ is the exceptional curve of  $S^{k+1}\to S^{k}$.
\end{enumerate}	
\end{enumerate}

\end{co}

\begin{ex} \label{eclatementsantican}  
We start with a singular cubic curve  with a single node $D^1_0\subset S^1:=\P^2(\C)$. We obtain a sequence $(S^k,D^k=\sum_{i=0}^{k-1} D^k_i)$ of rational surfaces with anti-canonical cycles  in the following way:
\begin{enumerate}
\item $S^2$ is obtained from $S^1$ by blowing up the double point of $D^1=D^1_0$, and $D^2=D^2_0+D^2_1$ where  $D^2_0$ is the exceptional curve, and $D^2_1$  is the strict transform of $D^1_0$.
\item $S^{k+1}$ is obtained from $S^{k}$ by blowing up a point belonging to the intersection of the strict transform of $D^1_0$ with the exceptional curve of $S^{k}\to S^{k-1}$, $D^k_0$ is the exceptional curve of  $S^{k+1}\to S^{k}$, and $D^{k+1}_{k}$ is the strict transform  of $D^1_0$.
\end{enumerate}
The type of $D^1$ is $[-7]$, and for $k\geq 2$ the type of $D^k$ is $[1,2,\dots 2, k-7]$.   Blowing up at suitable smooth points of $D^k$  we see that, for $1\leq k\leq 10$, any type $[d_0,\dots,d_{k-1}]\in {\cal T}_k$ is the type of an anti-canonical cycle in a smooth rational surface.
\end{ex}

\begin{ex} Consider the cycle $C=C_0+\cdots+C_{r-1}\subset X$ of $r\ge 1$ rational curves in $X$ of type $(3,2,\ldots,2)$, and $(Y,c)$ be the corresponding cusp. With the notations of section \ref{HZsection} we have $s=\sum_{i=0}^{r-1}(c_i-2)=1$ and the dual type is $[r+2]$, so it corresponds to a cycle $D$ consisting of a single singular rational curve with a single node and self-intersection $-r$.  Let $(Y^\smvee,d)$ be the cusp obtained by contracting such a cycle $D$ in a smooth surface $X^\smvee$.
By Wahl's theorem \ref{Wahl81}, a necessary condition for the cusp $(Y^\smvee,d)$ to be smoothable is   $r\le 10$, so for $r\ge 11$, $(Y^\smvee,d)$ is not smoothable. In fact by Theorem \ref{condsufflissage}, this cusp is smoothable if and only if $1\le r\le 10$. Note that $(Y,c)$ is  smoothable for any $r\in\N^*$.
\end{ex}

We can state now

\begin{thry} \label{condsufflissage} Let $X$ be a minimal class  VII surface,  $C\subset X$ be a cycle of rational curves with $C^2<0$,  $[c_0,\dots,c_{r-1}]$ be its type, and $(Y,c)$ be the singularity obtained by contracting $C$. 
\begin{enumerate}
\item If $\sum_{i=0}^{r-1} (c_i-2)\leq 10$ the singularity $(Y,c)$ is smoothable.
\item The condition $\sum_{i=0}^{r-1} (c_i-2)\leq 10$ is satisfied in the following two cases
\begin{enumerate}[(i)]
\item $r<b_2(X)$ and $b_2(X)\leq 11$.
\item $r=b_2(X)$ and $b_2(X)\leq 10$.
\end{enumerate}	
Therefore in these cases $(Y,c)$ is always smoothable.
\end{enumerate}
  \end{thry}

\begin{proof} (1) This follows from Theorem \ref{ThLooijConj},  Corollary \ref{LocGlobSmCo} and Example \ref{eclatementsantican}, taking into account  that the length of the dual of a type $[c_0,\dots,c_{r-1}]\in {\cal T}_r$ is 
$$s=\sum_{i=0}^{r-1} (c_i-2)=-C^2$$
 (see   section \ref{HZsection}). For (2)   note that $s<b_2(X)$ if $r<b_2(X)$, and $s=b_2(X)$ if  $r=b_2(X)$.
	
\end{proof}

\section{Smooth  deformations of singular contractions of class VII surfaces }

In this section we prove a general  smoothability criterion for the singular surface $Y$ obtained by contracting a cycle in a minimal class VII surface, and, in the smoothable case, we identify the smooth surfaces $Y'$ which appear as small deformations of $X$.

 Note that these results  are  valid for {\it unknown} minimal class VII surfaces: one just needs the existence of a cycle, not the existence of global spherical shell. In other words we do not assume that $X$ is a Kato surface.  
\begin{lm}\label{Lemmh0}
Let $X$ be minimal class VII surface, and $C\subset X$ be a  cycle of $r$ rational curves with $C^2<0$. 
\begin{enumerate}	
\item If $r< b_2(X)$ then $h^0({\cal K}_X(C)^{\otimes n})=0$ for any $n\in\N^*$.
\item If $r= b_2(X)$ then $X$ is a half Inoue surface, and
$$h^0({\cal K}_X(C)^{\otimes n})=\left\{
\begin{array}{ccc}
1 & \rm if & n\in 2\N\\
0& \rm if & n\in 2\N+1
\end{array}\right.\ .
$$
\end{enumerate}
\end{lm}
\begin{proof}
(1) Since $C$ is a cycle, we have ${\cal K}_X(C)_C=\omega_C\simeq {\cal O}_C$, so 
\begin{equation}\label{CycleId}
c_1({\cal K}_X)c_1({\cal O}(C))+c_1({\cal O}(C))^2=0\,.	
\end{equation}
Using \cite[Corollary 2.36, Remark 2.37]{DlJAMS} and $c_1(X)^2=-b_2(X)$, we obtain
$$c_1({\cal K}_X(C))c_1({\cal K}_X)=c_1({\cal K}_X)^2+c_1({\cal K}_X)c_1({\cal O}(C))=-b_2(X)-C^2=$$
$$=\left\{\begin{array}{ccc}
-r&\rm if &r<b_2(X)\\
0 &\rm if & r=b_2(X)	
\end{array}\right.\,.
$$
The claim follows now from \cite[Lemma 1.1.3]{NaToh}.\\ \\
(2) This is a consequence of Remark \ref{HalfInoue} (\ref{r=b2}).
\end{proof}

\begin{thry} \label{lissageglobal}
Let $X$ be a minimal class VII surface,  $C\subset X$ be a  cycle of $r$ rational curves with $C^2<0$, $[c_0,\ldots,c_{r-1}]$ be its type, and  $(Y,c)$ be the singular contraction	 of $(X,C)$.    Then $r\leq b_2(X)$ and
\begin{enumerate}
\item $Y$ is smoothable if and only if the  dual  $[d_0,\ldots, d_{s-1}]$ of  $[c_0,\ldots,c_{r-1}]$ is the type of an anti-canonical cycle in  a smooth rational surface which admits $\P^2$ as minimal model. This condition is always satisfied when $\sum_{i=0}^{r-1} (c_i-2)\leq 10$.  	
\item If $r<b_2(X)$, then  any smooth deformation $Y'$  of $Y$ is a rational surface with 
$$b_2(Y')=10+b_2(X)+C^2=10+r\,.$$
\item If $r=b_2(X)$, then	 $X$ is a half-Inoue surface, and  any smooth deformation $Y'$ of $Y$ is an Enriques surface.
\end{enumerate}
\end{thry}
\begin{proof} Since  $X$ admits a cycle $C$ of rational curves with $C^2<0$, it cannot be an Enoki surface, so the classes of the irreducible curves are linearly independent in $H_2(X,\Q)$. This implies $r\leq b_2(X)$.  \\
(1) This follows directly from Theorem \ref{condsufflissage} and Corollary \ref{LocGlobSmCo}.\\
\\
(2), (3) Let $\pi:X\to Y$ be the contraction map. Use  the exact sequence
$$0\to H^1(Y,{\cal O}_Y)\to  H^1(X,{\cal O}_X)\to H^0(Y,R^1\pi_*({\cal O}_X))\to H^2(Y,{\cal O}_Y)\to 0
$$
given by the Leray spectral sequence and the vanishing of $H^2(X, {\cal O}_X)$ \cite[Theorem 1.2, p. 478 ]{NaToh}. By Corollary \ref{h2Y} we have  $H^2(Y,{\cal O}_Y)=0$, so  the morphism $H^1(X,{\cal O}_X)\to H^0(Y,R^1\pi_*({\cal O}_X))$ is surjective. On the other hand, since $(Y,c)$ is an elliptic singularity, we have  $h^0(Y,R^1\pi_*({\cal O}_X))=1$, so the morphism 
$$H^1(X,{\cal O}_X)\to H^0(Y, R^1\pi_*({\cal O}_X))$$
is an isomorphism. Therefore 	$h^1(Y,{\cal O}_Y)=0$. Using  Theorem \ref{piTh} and Lemma \ref{Lemmh0} (or Serre duality) we obtain $h^0(\omega_Y)=0$. Let  now
$$Y\stackrel{i_0}{\longhookrightarrow} {\cal Y}\textmap{p} U	$$
 be a deformation of $Y$ with base $(U,u_0)$. Since $p$ is   flat, one can define an invertible sheaf $\omega$ on the total space ${\cal Y}$ whose restriction to any fibre $Y_u$ is the dualising sheaf $\omega_{Y_u}$ (see \cite[section 2.1.12, p. 51]{HLO}).  Applying the semicontinuity theorem to the sheaves ${\cal O}_{\cal Y}$, $\omega$ and assuming $U$ sufficiently small we get: 
 \begin{equation}\label{h1h0}
 \forall u\in U,\ h^1({\cal O}_{Y_u})=h^0(\omega_{Y_u})=0\,.
 \end{equation}
 \vspace{1mm}

Suppose first that $r<b_2(X)$. In this case a new application of Theorem \ref{piTh}, Lemma \ref{Lemmh0} and semicontinuity theorem gives
\begin{equation}
\forall u\in U,\ h^0(\omega_{Y_u}^{\otimes 2})=0  \,.	
\end{equation}
so any smooth fibre $Y'=Y_u$ has   $q(Y')=P_2(Y')=0$, so $Y'$ is a rational surface by Castelnuovo rationality criterion. Using the Noether formula, Theorem \ref{piTh} (\ref{omegaY})(b) and (\ref{CycleId}) we obtain 
\begin{equation}\label{comp}
\begin{split}
 2+b_2(Y')&=\chi(Y')=c_2(Y')=12\chi({\cal O}_{Y'})-c_1(Y')^2=12-c_1(\omega_{Y'})^2\\
&=12-c_1(\omega_{Y})^2=12-c_1({\cal K}_X(C))^2=12+b_2(X)+C^2\,.
\end{split}
\end{equation} 
The equality $b_2(X)+C^2=\#C$ follows from \cite[Corollary 2.36 p. 664]{DlJAMS}.
\\

Suppose now that $r=b_2(X)$. In this case  Theorem \ref{piTh} and Remark \ref{HalfInoue} (\ref{order2}) give
$$\omega_Y^{\otimes 2}=\pi_*({\cal K}_X(C)^{\otimes 2})=\pi_*({\cal O}_X)={\cal O}_Y\,.
$$
so    
\begin{equation}\label{h0=1h1=0}
h^0(\omega_Y^{\otimes 2})=1,\ 	h^1(\omega_Y^{\otimes 2})=0	\,.
\end{equation}
Using  \cite[Corollary 3.9, p. 122]{BS} it follows that,  choosing $U$ sufficiently small,   the sheaf $p_*(\omega^{\otimes 2})$ is free on $U$, and   the canonical map  
$$p_*(\omega^{\otimes 2})(u_0)\to H^0(\omega_{Y}^{\otimes 2})$$
is an isomorphism. Since $\dim(H^0(\omega_{Y}^{\otimes 2}))=1$ we have $p_*(\omega^{\otimes 2})\simeq {\cal O}_U$. Choose a nowhere vanishing section $\chi\in H^0(U,p_*(\omega^{\otimes 2}))$, and let $\tilde \chi$ be the corresponding section  of $\omega^{\otimes 2}$. The restriction of $\tilde \chi$ to  the fibre $Y$ is nowhere vanishing (because the restriction morphism $p_*(\omega^{\otimes 2})(u_0)\to H^0(\omega_{Y}^{\otimes 2})$ is an isomorphism, and $\omega_{Y}^{\otimes 2}$ is trivial), so the vanishing locus of $\Delta$ of $\tilde\chi$ does not intersect $Y$. Replacing $U$ by $U\setminus p(\Delta)$ and ${\cal Y}$   by $p^{-1}(U\setminus p(\Delta))$ if necessary, we may suppose that  $\omega^2$ is trivial on the total space ${\cal Y}$.

For a smooth fibre $Y'=Y_u$ we obtain ${\cal K}_{Y'}=\omega_{Y'}\simeq {\cal O}_{Y'}$, and by (\ref{h1h0}) we have $q(Y')=p_g(Y')=0$.  This shows that $Y'$ is an Enriques surface.

\end{proof}

The case when   $C$ is one of the two cycles of  an Inoue-Hirzebruch surface has been studied in \cite[section III.2]{Lo81}. Looijenga's results show that:

\begin{thry}\label{LoojThIH} Let $X$ be an Inoue-Hirzebruch surface with two cycles $C$, $D$ of rational curves,  let $\pi:(X,C)\to (Y,c)$ be the contraction on the singular surface obtained by contracting $C$, and $\pg:(X,C,D)\to (Z,c,d)$ be the contraction on the singular surface obtained by contracting $C$ and $D$.  Let
$$Z\stackrel{j_0}{\longhookrightarrow} {\cal Z}\textmap{p} V$$
be a universal deformation of $Z$ with base $(V,v_0)$,  let 
$$p_0:{\cal Z}_0\to V_0
$$
be the subfamily preserving the singularity $(Z,d)$,  and $\delta\subset {\cal Z}_0$ be the section defined by the cusps obtained as deformations of $d$.  Let 
$$Y\stackrel{\hat{j}_0}{\longhookrightarrow} \hat {\cal Z}_0\textmap{\hat p_0} V_0$$
 be the family obtained by minimal resolution along $\delta$.    Then
\begin{enumerate}
\item The germ $(\hat {\cal Z}_0,\hat{j}_0(c))$ is a semi-universal deformation of the singularity $(Y,c)$, in particular $\hat p_0$ has smooth fibres near $Y$ if and only if  this singularity is smoothable. 
\item The exceptional divisor of $ \hat {\cal Z}_0$ is a relative anti-canonical cycle.
\item Any smooth fibre $Y_v$ of $\hat p_0$ is a rational surface endowed with an anti-canonical cycle $D_v$ with the same type as $D$.	
\end{enumerate}
\end{thry} 
This result shows that if  $(Y,c)$ is smoothable, then the type  $[d_0,\dots,d_{s-1}]$ of $D$  is the type of an anti-canonical cycle in a rational surface; this was the motivation behind Looijenga's conjecture. Since now this conjecture is a theorem,  we obtain
\begin{co}
In the conditions and with the notations of Theorem 	\ref{LoojThIH}, the following conditions are equivalent:
\begin{enumerate}
	\item The singularity  $(Y,c)$ is smoothable.
	\item The type $[d_0,\dots,d_{s-1}]$ of $D$ is the type of an anti-canonical cycle in  a smooth rational surface which admits $\P^2$ as minimal model.
	\item $Y$ admits smooth deformations which are  smooth rational surfaces  admitting $\P^2$ as minimal model, and endowed with  anti-canonical cycles.
\end{enumerate}
\end{co}

\end{document}